\documentclass[12pt]{amsart}
\usepackage[margin = 1in]{geometry}
\usepackage{amssymb}
\usepackage{xcolor}
\usepackage{booktabs}
\usepackage{mathtools}
\usepackage{tikz,pgf}
\usetikzlibrary{decorations.text}
\usetikzlibrary{arrows,shapes}
\usepackage[group-separator={,}]{siunitx}
\usepackage{hyperref} 
\hypersetup{
    colorlinks,
    linkcolor={red!80!black},
    citecolor={blue!50!black},
    urlcolor={blue!80!black}
}

\newtheorem{theorem}{Theorem}[section]
\newtheorem{lemma}[theorem]{Lemma}

\newtheorem{corollary}[theorem]{Corollary}

\theoremstyle{definition}

\newtheorem{example}[theorem]{Example}
\newtheorem{remark}[theorem]{Remark}
\newtheorem{question}[theorem]{Question}

\newcommand{\ZZ}{\mathbb{Z}}
\DeclarePairedDelimiterX\set[1]{\lbrace}{\rbrace}{\def\given{\;\delimsize\vert\;}#1}
\DeclarePairedDelimiter{\floor}{\lfloor}{\rfloor}
\DeclarePairedDelimiter{\abs}{\lvert}{\rvert}

\newcommand{\card}[1]{\##1}
\newcommand{\length}{\ell}
\newcommand{\glength}{\length_g}
\newcommand{\sP}{\mathcal{P}}
\newcommand{\magma}{\textsc{Magma}}

\DeclareMathOperator{\Cay}{\Gamma}

\begin{document}

\title{An exploration of {Nathanson's} $g$-adic representations of integers} 
\author[G.~Bell]{Greg Bell}
\address{G.~Bell, Department of Mathematics and Statistics, The University of
  North Carolina at Greensboro, Greensboro, NC 27402, USA}  
\email{gcbell@uncg.edu}
\urladdr{\url{http://www.uncg.edu/~gcbell/}} 
\author[A.~Lawson]{Austin Lawson}
\address{A.~Lawson, Department of Mathematics and Statistics, The University of
  North Carolina at Greensboro, Greensboro, NC 27402, USA}  
\email{azlawson@uncg.edu}
\urladdr{\url{http://www.uncg.edu/~azlawson/}} 

\author[C.~Pritchard]{ Neil Pritchard}
\address{C.~Pritchard, Department of Mathematics and Statistics, The University
  of North Carolina at Greensboro, Greensboro, NC 27402, USA}  
\email{cnpritch@uncg.edu}
\urladdr{\url{http://www.uncg.edu/~cnpritch/}} 

\author[D.~Yasaki]{Dan Yasaki}
\address{D.~Yasaki, Department of Mathematics and Statistics, The University of
  North Carolina at Greensboro, Greensboro, NC 27402, USA}  
\email{d\_yasaki@uncg.edu}
\urladdr{\url{http://www.uncg.edu/~d_yasaki/}}

\begin{abstract}
We use Nathanson's $g$-adic representation of integers to relate
metric properties of Cayley graphs of the integers with respect to
various infinite generating sets $S$ to problems in additive number
theory.  If $S$ consists of all powers of a fixed integer $g$, we find
explicit formulas for the smallest positive integer 
of a given length.  This is related to finding the smallest positive
integer expressible as a fixed number of sums and differences of
powers of $g$.  We also consider $S$ to be the set of all powers of
all primes and bound the diameter of Cayley graph by relating it to
Goldbach's conjecture.
\end{abstract}

\maketitle

\section{Introduction}
Fix an integer $g \geq 2$.   Nathanson \cite{nathanson-gadic}
 introduces a $g$-adic representation of the integers $\ZZ$ in his investigation
 of number theoretic analogues of nets in metric geometry.  The
 $g$-adic representation of an integer provides a method of computing its length
 in a metric depending on $g$.
 Nathanson \cite{nathanson-diameter} considers more general sets of integers to
 investigate questions about the finiteness of the diameter of $\ZZ$.  In 
 this note, we continue these investigation in two different directions.  First,
 we find a formula for the smallest positive integer 
 with a given length (Theorems~\ref{thm:lambda-odd} and
 \ref{thm:lambda-even}).  Next, in Section~\ref{sec:CP} we consider a specific 
 set of integers constructed from primes.  We prove that $\ZZ$ has diameter $3$ or $4$ in the
 corresponding metric (Theorem~\ref{thm:diameter}), and it is $3$ assuming
 Goldbach's Conjecture. 

\section{Cayley graphs of the integers}
Let $G$ be a group.  Fix a generating set $S$ for $G$.  Then we can
construct a graph $\Gamma=\Gamma(G,S)$, known as the \emph{Cayley graph
corresponding to $G$ and $S$}, by taking 
the vertices of $\Gamma$ to be the elements of $G$ and connecting any two
vertices $g$ and $h$ by an edge whenever $gs=h$ for some element $s$ in $S$. We
view the graph $\Gamma$ as a metric space by setting the length of each
edge to be $1$ and taking the shortest-path metric on $\Gamma$; this procedure turns the
algebraic object $G$ into a geometric object $\Gamma$. 
In this paper, we investigate different choices of infinite
generating sets $S$ when $G=\mathbb{Z}$.

We are primarily interested in generating sets that are closed under additive
inverses and are closed under taking powers. The simplest such generating set is
the collection 
\[S_g=\set{\pm 1,\pm g,\pm g^2,\pm g^3,\ldots}.\] 
We denote the Cayley graph $\Gamma(\mathbb{Z},S_g)$ by $C_g$. Edges in the graph
$C_g$ connect each vertex 
to infinitely many other vertices, see Figure~\ref{fig:balls}.
More generally, let $P$ be a subset of positive integers, and consider
generating sets of the form 
\[S_P=\bigcup_{g\in P}S_g.\] 
Let $C_P = \Gamma(\ZZ,S_P)$ denote the corresponding Cayley graph.  

The study of these graphs leads to an interesting interplay
between the geometry of the graph and problems in additive number theory. For
example, in $C_2$ we can ask for 
the value of the smallest $n>0$ (in the usual ordering of the integers) at
distance $d$ from $0$.  This is related to the problem of finding the smallest
integer that can be expressed as sums and differences of exactly $d$ powers of
$2$.  Looking 
at Figure~\ref{fig:balls}, one can see that $3$ 
is the smallest positive number that is at distance $2$ from $0$, since $3 = 2^0
+ 2^1$; extrapolating
this figure further, one can verify that $11$ is the smallest positive integer
at distance $3$ from $0$, since $11 = 2^0 + 2^1 + 2^3$.  We investigate this
problem in general for $g > 1$ in Section~\ref{sec:Cg}.  

In Section~\ref{sec:CP}, we investigate $C_P$ for more general subsets $P$ of
positive integers.  For such graphs, questions about the
diameter are already interesting and difficult
\cite{nathanson-diameter, vasilevska-exploring}.  
We use a covering congruences result of Cohen and Selfridge
\cite[Theorem~2]{cohen-selfridge} together with Helfgott's proof
\cite[Main Theorem]{helfgott} of the ternary Goldbach's conjecture to show  
that when $P$ is the set of all primes, the diameter of $C_P$  
 is either $3$ or $4$; moreover if Goldbach's conjecture holds,
it is $3$. We also conduct numerical investigations to narrow the search for the
smallest positive length-$3$ integer in this graph, refining results of Cohen
and Selfridge \cite{cohen-selfridge} and Sun \cite{sun-2000}.

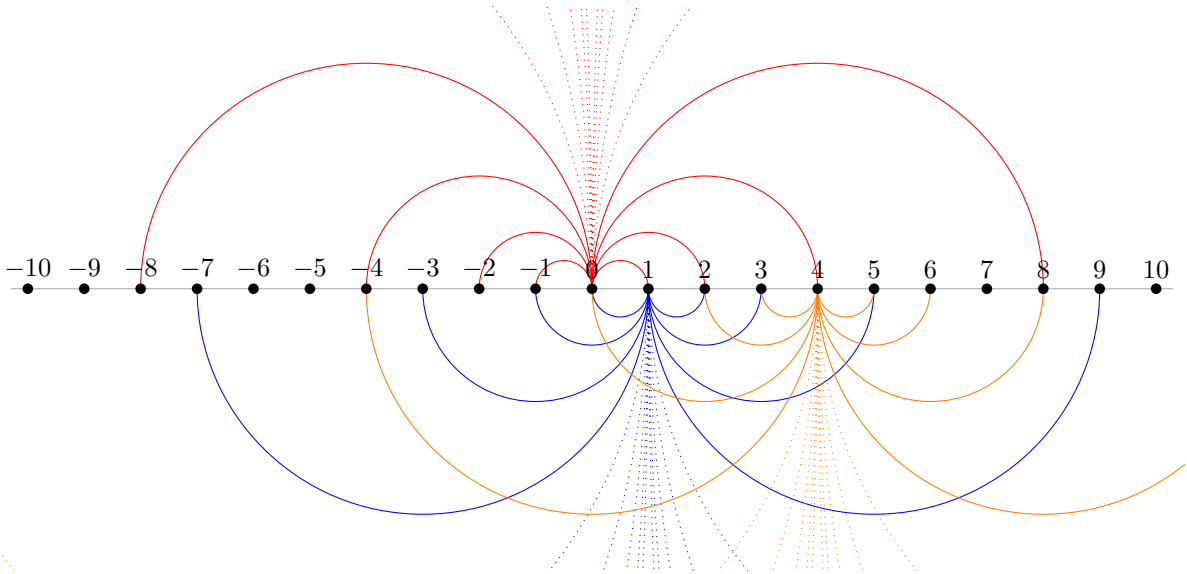
\begin{figure}
\begin{tikzpicture}[x=.75cm, y=.75cm]
\clip (-10.5,-5) rectangle (10.5,5);
\draw [gray!70, very thin](-10.3,0)--(10.3,0); 
%
\foreach \x in {1,2,4,8}{
\draw[color=red] (\x,0) arc [start angle=0, end angle = 180, radius =.5*abs(\x)];};
\foreach \x in {-1,-2,-4,-8}{
\draw[color=red] (0,0) arc [start angle=0, end angle = 180, radius =.5*abs(\x)];};
\foreach \x in {16,32,64,128,256}{
\draw[color=red, dotted] (0,0) arc [start angle=0, end angle = 180, radius =.5*\x];};
\foreach \x in {16,32,64,128,256}{
\draw[color=red, dotted] (\x,0) arc [start angle=0, end angle = 180, radius =.5*\x];};
%
\foreach \x in {1,2,4,8}{
\draw[color=blue] (1,0) arc [start angle=180, end angle = 360, radius =.5*abs(\x)];};
\foreach \x in {-1,-2,-4,-8}{
\draw[color=blue] (\x+1,0) arc [start angle=180, end angle = 360, radius =.5*abs(\x)];};
\foreach \x in {16,32,64,128,256}{
\draw[color=blue, dotted] (1,0) arc [start angle=180, end angle = 360, radius =.5*\x];};
\foreach \x in {-16,-32,-64,-128,-256}{
\draw[color=blue, dotted] (\x+1,0) arc [start angle=180, end angle = 360, radius =-.5*\x];};
%
\foreach \x in {1,2,4,8}{
\draw[color=orange] (4,0) arc [start angle=180, end angle = 360, radius =.5*abs(\x)];};
\foreach \x in {-1,-2,-4,-8}{
\draw[color=orange] (\x+4,0) arc [start angle=180, end angle = 360, radius =.5*abs(\x)];};
\foreach \x in {16,32,64,128,256}{
\draw[color=orange, dotted] (4,0) arc [start angle=180, end angle = 360, radius =.5*\x];};
\foreach \x in {-16,-32,-64,-128,-256}{
\draw[color=orange, dotted] (\x+4,0) arc [start angle=180, end angle = 360, radius =-.5*\x];};
%
\foreach \t in {-10,-9,-8,-7,-6,-5,-4,-3,-2,-1,0,1,2,3,4,5,6,7,8,9,10}
\fill [shift={(\t,0)}] circle (2pt) (0,0) node [above] {\footnotesize$\t$};
\end{tikzpicture}
\caption{Some edges emanating from $0$ in red, from $1$ in blue, and from $4$ in
  orange in the graph $C_2$. Each vertex is incident with infinitely many
  edges. Observe that the distance from $0$ to $3$ is $2$.} 
\label{fig:balls}
\end{figure}

\section{Metric properties of \texorpdfstring{$C_g$}{Cg}} \label{sec:Cg}
Let $g>0$ be an integer, and let $C_g = \Cay (\ZZ,S_g)$ be the
Cayley graph of $\ZZ$ with the generating set $S_g=\set{\pm g^i\given i\in \ZZ_{\ge
  0}}$.  Let $d_g = d_{S_g}$ denote the corresponding edge-length metric. We
denote the distance $d_g(0,n)$ by $\ell_g(n)$ and refer to this as the
\emph{$g$-length} of $n$. 

The following theorems of Nathanson \cite{nathanson-gadic} give a method of
computing $g$-length in $C_g$. 

\begin{theorem}[{\cite[Theorem 6]{nathanson-gadic}}] \label{thm:gadic-odd}
Let $g$ be an odd integer, $g \geq 3$.  Every integer $n$ has a unique
representation in the form  
\[n = \sum_{i = 0}^\infty \epsilon_i g^i\]
such that 
\begin{enumerate}
\item  $\epsilon_i \in \set{0, \pm 1, \pm 2, \dots, \pm (g-1)/2}$ for all
  nonnegative integers $i$, 
\item $\epsilon_i  \neq 0$ for only finitely many nonnegative integers $i$.
\end{enumerate}
Moreover, $n$ has $g$-length 
\[\glength(n) = \sum_{i = 0}^\infty\ \abs{\epsilon_i}.\]
\end{theorem}

\begin{theorem}[{\cite[Theorem 3]{nathanson-gadic}}] \label{thm:gadic-even}
Let $g$ be an even positive integer.  Every integer $n$ has a unique
representation in the form  
\[n = \sum_{i = 0}^\infty \epsilon_i g^i\]
such that 
\begin{enumerate}
\item  $\epsilon_i \in \set{0, \pm 1, \pm 2, \dots, \pm \frac{g}{2}}$ for all nonnegative integers $i$,
\item $\epsilon_i  \neq 0$ for only finitely many nonnegative integers $i$, 
\item if $\abs{\epsilon_i}  = \frac{g}{2}$, then $\abs{\epsilon_{i + 1}} < \frac{g}{2}$ and
  $\epsilon_i \epsilon_{i + 1} \geq 0$. 
\end{enumerate}
Moreover, $n$ has $g$-length 
\[\glength(n) = \sum_{i = 0}^\infty \abs{\epsilon_i}.\]
\end{theorem}

For any integer $n$, Theorems~\ref{thm:gadic-odd} and \ref{thm:gadic-even} 
 give a unique $g$-adic expression for $n$ that realizes a
geodesic path from $0$ to $n$ in $C_g$.  Thus there is an $N > 0$ such that $n = \sum_{i =
  0}^N \epsilon_i g^i$, $\epsilon_N \neq 0$, and $\glength(n) = \sum_{i =
  0}^\infty\ \abs{\epsilon_i}$.  We call $n = \sum_{i = 0}^N \epsilon_i g^i$
the \emph{minimal $g$-adic expansion}, and denote it by  
\[[n]_g = [\epsilon_0, \epsilon_1, \dots, \epsilon_N].\]  

\begin{remark}
It is not the case that there is a unique geodesic path from $0$ to
$n$.  For example, $11 = 1 + 2^1 + 2^3 = -1 - 2^2 + 2^4$.
\end{remark}

It is interesting to look at how $\glength(n)$ varies as a function of $g$.  See
Figure~\ref{fig:20233509}.  We chose a random number $n = \num{20233509}$, and
produced a plot of $y = \glength(n)$ for a range of values for $g$.  For $g$
sufficiently large, we  have $\glength(n) = n$, but it appears that interesting
things happen along the way. 

\begin{figure} 
\begin{tabular}{cc}
\includegraphics[width = 0.4\textwidth]{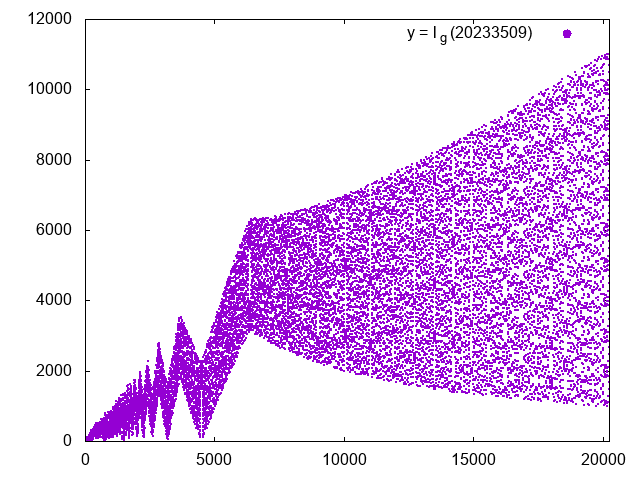} &
\includegraphics[width = 0.4\textwidth]{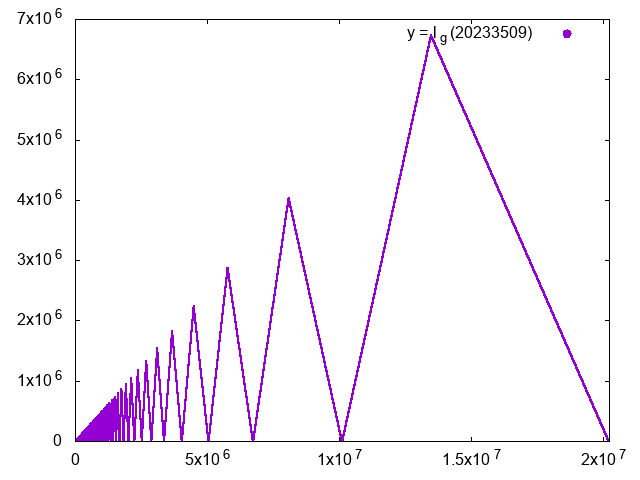}
\end{tabular}
\caption{Plots of $y = \glength(20233509)$ as a function of $g$.} \label{fig:20233509}
\end{figure} 

\begin{example}
 The minimal $5$-adic expansion of $46$ is $[46]_5  =[1, -1, 2]$, so $46 = 1 - 5
 + 2\cdot 5^2$, and $\ell_5(46) = 1 + 1 + 2 = 4$.   
\end{example}

We denote by $\lambda_{g}(h)$ the smallest positive integer of $g$-length $h$ in
$C_g$. We find an explicit formula for $\lambda_g$ in
Theorems~\ref{thm:lambda-odd} and ~\ref{thm:lambda-even} below using
Nathanson's $g$-adic representation \cite{nathanson-gadic} of positive integers.
The first few values are tabulated in Table~\ref{tab:lambda-prime}.  We remark
that the values of $\lambda_2$ show up in  
The On-Line Encyclopedia of Integer Sequences (\href{https://oeis.org}{OEIS}) as
\href{https://oeis.org/A007583}{A007583}, and the values of $\lambda_3$ show up
as \href{https://oeis.org/A007051}{A007051}.  The sequences of values for
$\lambda_p$ for other primes $p$ did not appear, so the second author has added
them. 
As an example, we chose the prime 19; 
Figure~\ref{fig:lambda_19} shows the  integers less than \num{10000} and their
$19$-length, together with the graph of $y = \lambda_{19}(x)$. 

\begin{table}
\caption{First few values of $\lambda_p(k)$ for primes $p < 30$.}
\label{tab:lambda-prime}
\begin{tabular}{rrrrrrrrrrr}
\toprule
$k$ & 2& 3 & 5 & 7 & 11 & 13 & 17 & 19 & 23 & 29\\
\midrule
1&1&1&1&1&1&1&1&1&1&1\\
2&3&2&2&2&2&2&2&2&2&2\\
3&11&5&3&3&3&3&3&3&3&3\\
4&43&14&8&4&4&4&4&4&4&4\\
5&171&41&13&11&5&5&5&5&5&5\\
6&683&122&38&18&6&6&6&6&6&6\\
7&2731&365&63&25&17&7&7&7&7&7\\
8&10923&1094&188&74&28&20&8&8&8&8\\
9&43691&3281&313&123&39&33&9&9&9&9\\
10&174763&9842&938&172&50&46&26&10&10&10\\
11&699051&29525&1563&515&61&59&43&29&11&11\\
12&2796203&88574&4688&858&182&72&60&48&12&12\\
13&11184811&265721&7813&1201&303&85&77&67&35&13\\
14&44739243&797162&23438&3602&424&254&94&86&58&14\\
15&178956971&2391485&39063&6003&545&423&111&105&81&15\\
16&715827883&7174454&117188&8404&666&592&128&124&104&44\\
17&2863311531&21523361&195313&25211&1997&761&145&143&127&73\\
18&11453246123&64570082&585938&42018&3328&930&434&162&150&102\\
19&45812984491&193710245&976563&58825&4659&1099&723&181&173&131\\
20&183251937963&581130734&2929688&176474&5990&3296&1012&542&196&160\\
\bottomrule
\end{tabular}
\end{table}

\begin{figure}
\includegraphics[width=0.5\textwidth]{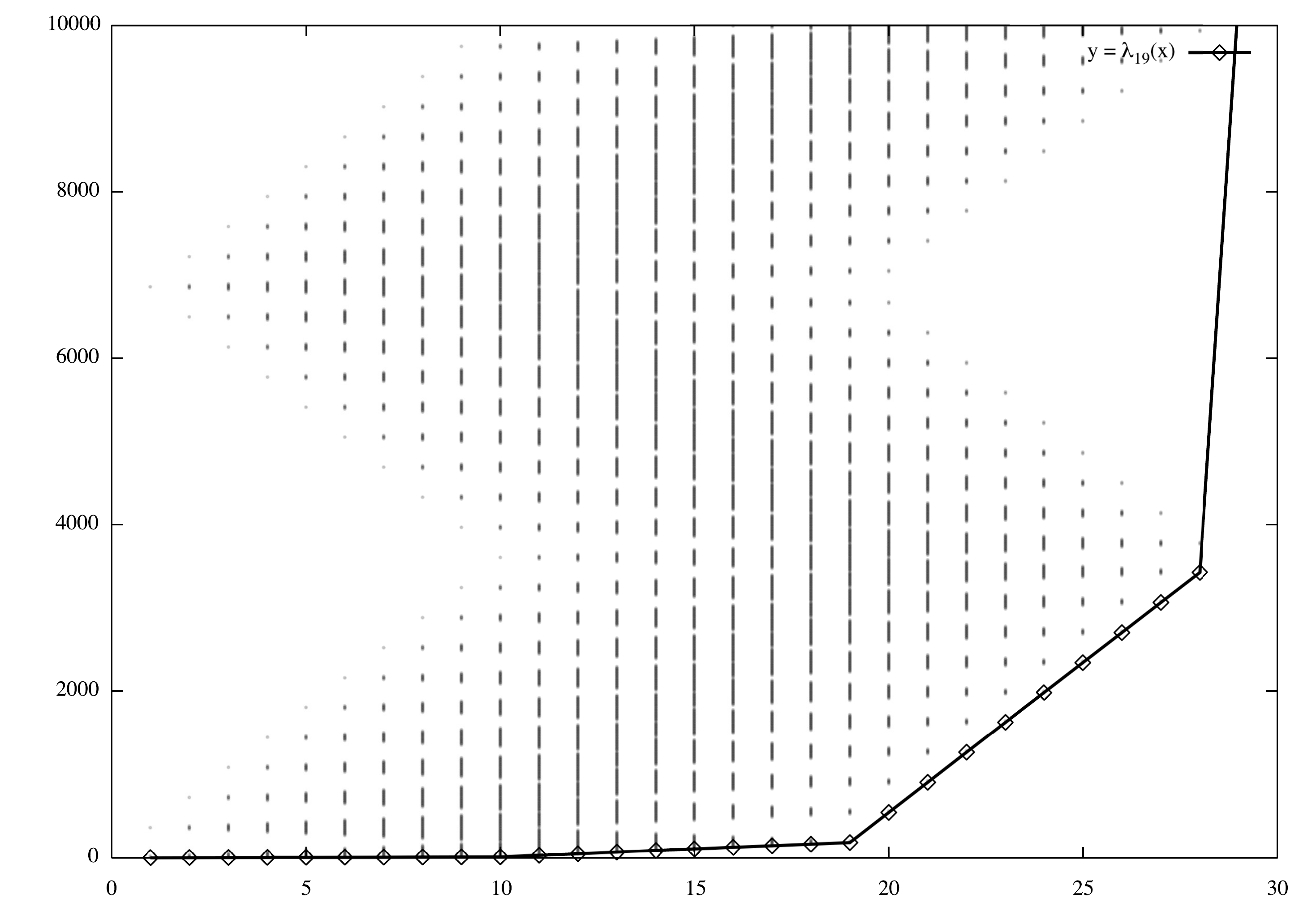}
\caption{The integers up to $\num{10000}$ whose $19$-lengths are $x$ are shown
  together with the graph of $y = \lambda_{19}(x)$.} \label{fig:lambda_19} 
\end{figure}

The following two theorems give an explicit formula for $\lambda_g$.
First, we need a preliminary lemma that relates the digits in
the minimal $g$-adic expansion of an integer to its size.

\begin{lemma}\label{lem:leading-coefficient}
  Let $g > 1$ be an integer.  Let $m$ and $n$ be integers
  with minimal $g$-adic expansions $[n]_g = [n_0, n_1, \dots, n_N]$
  and $[m]_g = [m_0, m_1, \dots, m_M]$. Set $n_i = 0$ for
  $i > N$ and $m_i = 0$ for $i > M$. Let $t \geq 0$ be the 
  largest integer such 
  that $n_i = m_i$ for all $i > t$.  Then $n > m$ if and only if $n_t
  > m_t$. 
\end{lemma}
\begin{proof}
By subtracting $s = \sum_{i = t + 1}^\infty n_i g^i$, we can assume without
loss of generality that $n = \sum_{i = 0}^ t  n_i g^i$ and $m =
\sum_{i = 0}^ t  m_i g^i$, so that $N = M = t$.  Since $-n = \sum_{i = 0}^ N
(-n_i) g^i$ and similarly for $m$, it suffices consider positive integers $m$
and $n$.  Relabel if necessary to assume without loss of generality that $n_N >
m_N$.  Then 
\[n - m = \sum_{i = 0}^ N  (n_i - m_i) g^i  \geq  g^N + \sum_{i = 0}^{N-1}
(n_i - m_i) g^i.\]
Thus it suffices to show 
\begin{equation}\label{eq:bound}
\sum_{i = 0}^{N - 1} (n_i - m_i)g^i < g^N.
\end{equation}
Let $n'= \sum_{i = 0}^{N - 1} n_i g^i$, and let $m' = \sum_{i =
  0}^{N - 1} m_ig^i$. The left side of \eqref{eq:bound} is maximized when $n'$
is positive and as large as possible, and $m'$ is negative and as small as possible, 
in which case we claim that $n_i-m_i\le g-1.$ Thus, it suffices to show~\eqref{eq:bound} 
in this case.

Suppose $g$ is odd.  Let $b = \frac{g - 1}{2}$.  Theorem~\ref{thm:gadic-odd}
implies $m_i$ and 
$n_i$ are in $\set{0, \pm 1, \dots, \pm b}$, so
\begin{equation}\label{eq:oddbound}
n_i - m_i \leq b + b = g - 1.
\end{equation}

Suppose $g$ is even.  Let $b = \frac{g}{2}$.
Theorem~\ref{thm:gadic-even} implies $m_i$ and 
$n_i$ are in $\set{0, \pm 1, \dots, \pm b}$.  Furthermore, 
if a digit $\abs{\epsilon_i} = b$, then $\abs{\epsilon_{i + 1}} <
b$ and $\epsilon_i\epsilon_{i + 1}
\geq 0$.  Note that $m'_{N - 1} \neq -b$, since $m > 
0$.  It follows that the smallest $m'$ can be is when the minimal
$g$-adic expansion of $m'$ alternates $-b$ and $-(b - 1)$; the largest
$n'$ can be is when the minimal $g$-adic 
expansion of $n'$ alternates between $b - 1$ and $b$.
Thus \[[m']_g \geq  [\dots, -b, -(b - 1)] \quad \text{and} \quad 
[n']_g \leq [\dots, b - 1, b],\]
so 
\begin{equation}\label{eq:evenbound}
  n_i - m_i \leq b + (b - 1) = g - 1.
\end{equation}

Thus, we have
\[\sum_{i = 0}^{N - 1} (n_i - m_i)g^i \leq \sum_{i = 0}^{N - 1} (g
- 1)g^i = (g
- 1) \sum_{i = 0}^{N - 1} g^i = (g
- 1)\left(\frac{g^N - 1}{g - 1}\right) = g^N - 1,\]
so \eqref{eq:bound} follows.
\end{proof}

\begin{theorem} \label{thm:lambda-odd}
Let $g > 1$ be an odd integer, and let $k>0$ be an integer.  Let $q =
\floor{\frac{2k}{g-1}}$, and let $r = k \bmod{\frac{g - 1}{2}}$ so that $k =
q(\frac{g - 1}{2}) + r$.  Let
\begin{align*}
A &= \begin{cases}
 \frac{g-1}{2} & \text{if $r = 0$,}\\
 -\left(\frac{g-1}{2}\right) & \text{otherwise;}\\
\end{cases} & 
B &= \begin{cases}
 0 & \text{if $r = 0$,}\\
 r & \text{otherwise;}\\
\end{cases}
\end{align*}
Then 
\[\lambda_g(k) = \frac{1 - g^{q - 1}}{2} + Ag^{q - 1} + Bg^q.\]
\end{theorem}

\begin{proof}  Let $b = \frac{g - 1}{2}$, and let $n = \frac{1 - g^{q - 1}}{2} + Ag^{q - 1} + Bg^q$.
     A straightforward 
  computation shows that the minimal $g$-adic expansion of $n$ is given by  
\begin{equation}\label{eq:odd}
[n]_g = 
\begin{cases}
\underbrace{[-b, -b, \dots, -b, b]}_{\text{$q$ digits}}  &\text{if $r
  = 0$,}\\
\underbrace{[-b,-b, \dots,
    -b, r]}_{\text{$q + 1$ digits}} & \text{otherwise}.
\end{cases}
\end{equation}  

First, we show $\glength(n)
= k$.  There are two cases to consider.
If $r = 0$, we have $q =
\floor{\frac{2k}{g-1}} = \frac{2k}{g-1}$.  It follows that 
\[\glength(n) = b q = \left(\frac{g-1}{2}\right)
\left(\frac{2k}{g-1}\right)= k,\] 
as desired.  If $r\neq 0$, then  \[\glength(n) = bq + r = k,\] by
construction.   

Finally, we show that $n$ is the smallest positive integer with this
property.  Suppose $m < n$ 
is a positive integer.  Let $[m]_g =
[m_0, m_1, \dots, m_M]$ be the 
minimal $g$-adic expansion of $m$, and  let $[n]_g = [n_0, n_1, \dots,
  n_N]$ be the minimal $g$-adic expansion given in \eqref{eq:odd}.
Set $n_i = 0$ for $i > N$, and set $m_i = 0$ for $i > M$.  Let $t$ be
the maximal index such that  
$m_t \neq n_t$.  By Lemma~\ref{lem:leading-coefficient},  we must have $m_t
< n_t$.  By Theorem~\ref{thm:gadic-odd}, we have $\abs{m_j} \leq
b$, for all $j$.  Since $n_j = -b$ for $0 \leq j \leq N - 1$,
we cannot have $t \leq N - 1$.  Since $0 \leq m < n$, this implies
$m_N < n_N$.  Furthermore, $m_j = n_j = 0$ for $j > N$.  It follows
that $\glength(m) < \glength(n)$.
\end{proof}

\begin{theorem}\label{thm:lambda-even}
Let $g > 1$ be an even integer, and let $k > 0$ be an integer.  Let $r = k
\bmod{g - 1}$.  Define integers $q$, $A$, and $B$ by  
\begin{align*}
q &= \begin{cases}
\floor{\frac{k}{g - 1}} - 1& \text{if $r = 0$},\\
\floor{\frac{k}{g - 1}} & \text{otherwise};
\end{cases}\\
A &= \begin{cases}
\frac{g}{2} & \text{if $r = 0$ or $r > \frac{g}{2}$},\\
r & \text{otherwise;}
\end{cases} \\
B &= \begin{cases}
\frac{g}{2} - 1 & \text{if $r = 0$,}\\
r - \frac{g}{2} & \text{if $r > \frac{g}{2}$,}\\
0 & \text{otherwise.}
\end{cases}
\end{align*}
Then 
\[\lambda_g(k) = \frac{g(1 - g^{2q})}{2  (1 + g)} + Ag^{2q} + Bg^{2q
            + 1}.\]
\end{theorem}
\begin{proof}
Let $b = \frac{g}{2}$, and let  \[n = \frac{g(1 - g^{2q})}{2  (1 + g)} + Ag^{2q} + Bg^{2q
            + 1}.\]
A straightforward computation shows
\begin{equation}\label{eq:even}
[n]_g  = 
\begin{cases}
\underbrace{[-b,-(b-1),-b,-(b-1),\ldots,-b,-(b-1),b,b-1]}_{\text{$2q + 2$
    digits}} & \text{if $r = 0$,}\\  
\underbrace{[-b,-(b-1),-b,-(b-1),\ldots,-b,-(b-1),b,r-b]}_{\text{$2q + 2$
    digits}} & \text{if $r > b$,}\\  
\underbrace{[-b,-(b-1),-b,-(b-1),\ldots,-b,-(b-1),r]}_{\text{$2q + 1$
    digits}} & \text{otherwise.}\\  
\end{cases}\end{equation}

First, we show $\ell_g(n) = k$.  There are three cases to consider.  Suppose $r
= 0$.  Then $g-1$ divides $k$, so  
\[q = \floor*{\frac{k}{g - 1}} - 1 = \frac{k}{g - 1} -1.\] Then 
\begin{align*}
  \ell_g(n) &= q(2b - 1) + b + (b - 1)\\
&= \left(\frac{k}{g - 1} - 1\right)(g - 1) + \frac{g}{2}
  + \left(\frac{g}{2} - 1\right)\\ 
&= k.
\end{align*}
If $r > \frac{g}{2}$, then 
\begin{align*}
  \ell_g(n) &= q(2b - 1) + b + (r - b)\\
&= \floor*{\frac{k}{g - 1}} (g - 1) + r\\
&= k.
\end{align*}
If $1 \leq r \leq b$, then 
\begin{align*}
  \ell_g(n) &= q(2b - 1) + r\\
&= \floor*{\frac{k}{g - 1}} (g - 1) + r\\
&= k.
\end{align*}

Finally, we show that $n$ is the smallest positive integer with this
property.  Suppose $m < n$ is a positive integer.  Let $[m]_g= 
[m_0, m_1, \dots, m_M]$ be the 
minimal $g$-adic expansion of $m$, and let $[n]_g = [n_0, n_1, \dots,
  n_N]$ be the minimal $g$-adic expansion given in \eqref{eq:even}.
Let $t$ be the maximal index such that 
$m_t \neq n_t$.  By Lemma~\ref{lem:leading-coefficient},  we have $m_t
< n_t$.  By Theorem~\ref{thm:gadic-even}, we have $\abs{m_j} \leq b$,
and if $\abs{m_j} = b$, then $\abs{m_{j + 1}} < b$ and $m_j m_{j + 1}
\geq 0$, for all $j$.  From \eqref{eq:even}, we cannot have $t \leq 2q
- 1$.  Then $\abs{m_j} \leq n_j$ for $0 \leq j \leq N -
1$, $m_j = 0$ for $j > N$, and $m_t < n_t$.  Thus $\glength(m) <
\glength(n)$. 
\end{proof}

\section{Metric properties of \texorpdfstring{$C_P$}{CP}}\label{sec:CP}
Let $P$ be a set of positive integers.   Let $C_P = \Cay(\ZZ, S_P)$ denote the Cayley graph
of $\ZZ$ with the generating set   
\[S_P = \bigcup_{a \in P}\set{\pm a^i \given  i \in \ZZ_{\ge 0}}.\]

We give $C_P$ the edge-length metric $d_{S_P}$, and use $\ell_P(n)$ to denote
the $P$-length of $n$ in the metric $d_{S_P}$, i.e., $\ell_P(n)=d_{S_P}(0,n)$.  
The $P$-length function is much more subtle when $\#P > 1$.

\begin{question}\label{q:lambda}
Let $P$ be a subset of primes.   Let $\lambda_{P}(h)$ denote the smallest positive
integer of $P$-length $h$ in $C_P$. Compute the function $\lambda_P(h)$.  
\end{question}

There are partial results addressing Question~\ref{q:lambda} when $\card{P} <
\infty$.  Hadju and Tijdeman \cite{hajdu-tijdeman2011} prove that
$\exp(ck)<\lambda_P(k)<\exp((k \log k)^C)$, with some constant $c$ depending on
$P$ and an absolute constant $C$.  

Nathanson \cite{nathanson-diameter} gives a  class of generating sets for $\ZZ$
whose arithmetic diameters are infinite. 

\begin{theorem}[{\cite[Theorem 5]{nathanson-diameter}}]\label{thm:nathanson-diameter}
If $P$ is a finite set of positive integers,  then $C_P$ has infinite diameter.
\end{theorem}

On the other hand, for infinite $P$ the diameter of $C_P$ may be finite.  The
\emph{ternary Goldbach conjecture} states that every odd integer $n$ greater
than $5$ can be written as the sum of three primes.  Helfgott's proof
\cite[Main Theorem]{helfgott} of this implies if $\sP$ is the set of all primes,
then $C_{\sP}$ is at most  $4$. 

\begin{theorem}\label{thm:diameter}
Let $\sP$ be the set of all primes.  The diameter of $C_\sP$ is $3$ or $4$.  
\end{theorem}
\begin{proof}
It is easy to see that $\length_\sP(n) = 1$ for $n \in \set{1, 2, 3, 4, 5}$.
Helfgott \cite[Main Theorem]{helfgott} proves that every odd integer greater
than  $5$ can be written as the sum of three primes.  Since every even integer
greater than $4$ can be expressed as $1$ less than an odd integer greater than
$5$, we have that $\length_\sP(n) \leq 4$ for all $n \in \ZZ$.   

Since not every integer is a prime power, the diameter of  $C_\sP$ is at least
$2$.  To show that the diameter is not $2$, it suffices to produce an integer
that is not a prime power and cannot be expressed as the sum or difference of
prime powers, where the prime power $p^0 = 1$ is allowed.   Such integers are
surprisingly hard to find.  First note that the Goldbach conjecture asserts that
every even integer greater than $2$ can be expressed as the sum of two primes.
This has been computationally verified integers less than $4 \cdot 10^{18}$
\cite{silva-herzog-pardi}.  It follows that $\ell_\sP(n) \leq 2$ for even
integers $n < 4 \cdot 10^{18}$.  Thus a search for an integer of $P$-length $3$
should be restricted to odd integers.  An odd integer $M$ is $P$-length $3$ if 
\begin{enumerate}
\item $M$ is not prime power; 
\item $\abs{M \pm 2^n}$ is not prime power for all $n \ge 0$. \label{it:pp}
\end{enumerate}

Cohen and Selfridge \cite[Theorem 2]{cohen-selfridge} use covering congruences
to prove the existence of an infinite family of integers $M$ satisfying item
\eqref{it:pp} and give an explicit  94-digit example of such an integer.
 Sun \cite{sun-2000} adapts their work  to produce a much smaller
example.  Specifically, let 
\[M = 47867742232066880047611079 ,\quad \text{and let} \quad N =
66483084961588510124010691590.\]   Sun proves that if $x \equiv  M\bmod{N}$, 
then $x$ is not of the form $\abs{p^a \pm q^b}$ for any primes $p,q$ and
nonnegative integers $a,b$\footnote{The modulus $N$ given by Sun \cite{sun-2000}
  is incorrectly written as $66483034025018711639862527490$.}.   
 We use  Atkin and Morain's ECPP (Elliptic Curve Primality Proving) method
 \cite{atkin-morain} implemented by Morain in \magma\ \cite{magma} to look in
 this congruence class for an element that is provably not a prime power.   We
 find that $M$ and $M + N$ are prime, but  
\begin{align*}
M + 2N &= 133014037665409087128068994259\\
&= 23\cdot 299723\cdot 19295212676140402555471
\end{align*}
is not a prime power.  Thus $\length_\sP(M + 2N) = 3$, and the result follows.
\end{proof}

\begin{remark}
Assuming Goldbach's conjecture, the diameter of $C_\sP$ is $3$.  
\end{remark}

It is still an open problem to find the smallest integer  $n$ that is not of the
form $\abs{p^a \pm q^b}$, for any primes $p,q$ and nonnegative integers $a,b$
\cite[A19]{guy}.    Explicit computations \cite{cohen-selfridge, sun-2000} show
that the smallest such integer must be larger than $2^{25}$.  Such elements, if
not prime powers, would have $\sP$-length $3$.  We have extended slightly their
computation and confirmed that $\length_\sP(n) < 3$ for all $n < \num{58164433}
\approx 2^{25.79}$.   For  
\[n = 58164433  = 4889 \cdot 11897,\]
we could not show $\length_\sP(n) = 2$.  It is possible that this integer is the
smallest positive integer of $\sP$-length $3$. 

\begin{corollary} Let $P$ be a subset of the natural numbers containing all but finitely many primes. Then, $C_P$ has finite diameter. 
\end{corollary}

\begin{proof}
It is enough to consider the case $P=\sP \setminus S$, where $S$ is finite. 
Let $R = \max_{p \in S}
\set{\length_{\sP'}(p)}$.

First note that if $p$ is a prime in $S$, then $\length_{\sP'}(p) \leq R$.  If
$p$ is a prime not in $S$, then by the ternary Goldbach conjecture \cite[Main
  Theorem]{helfgott} we have $\length_{\sP'}(p) \leq 3$.  Thus
$\length_{\sP'}(p) \leq \max\set{R,3}$ for any prime $p$. 

Since every even integer is one less than an odd integer, it suffices to show
that $\length_{\sP'}(n) \leq 3\max\set{R,3}$ for every positive odd integer
$n$. By the ternary Goldbach conjecture, every odd integer $n > 5$ can be
expressed as the sum of three primes.  Let $n = p + q + r > 5$ be an odd integer
for some primes $p, q, r$.  Then  
\[
\length_{\sP'}(n) \leq \length_{\sP'}(p) + \length_{\sP'}(q) + \length_{\sP'}(r)
\leq \max\set{R,3} + \max\set{R,3} + \max\set{R,3},\]
and the result follows.
\end{proof}
\providecommand{\bysame}{\leavevmode\hbox to3em{\hrulefill}\thinspace}
\providecommand{\MR}{\relax\ifhmode\unskip\space\fi MR }
\providecommand{\MRhref}[2]{%
  \href{http://www.ams.org/mathscinet-getitem?mr=#1}{#2}
}
\providecommand{\href}[2]{#2}

\end{document}